\documentclass[10pt]{amsart}

\usepackage{amssymb,amsmath,amsthm}
\usepackage[arc,all]{xy}
\usepackage{eucal}

\usepackage{color}

\theoremstyle{plain}

\newtheorem{thm}[subsection]{Theorem}
\newtheorem{prop}[subsection]{Proposition}

\newtheorem{convention}[subsection]{Convention}

\theoremstyle{definition}
\newtheorem{defn}[subsection]{Definition}

\theoremstyle{remark}
\newtheorem{rem}[subsection]{Remark}

\numberwithin{equation}{section}

\newcommand{\sSet}{{ \mathsf{sSet} }}

\newcommand{\Mod}{{ \mathsf{Mod} }}
\newcommand{\ModR}{{ \mathsf{Mod}_\capR }}

\newcommand{\Alg}{{ \mathsf{Alg} }}

\newcommand{\LL}{{ \mathsf{L} }}
\newcommand{\RR}{{ \mathsf{R} }}

\newcommand{\TQ}{{ \mathsf{TQ} }}

\newcommand{\K}{{ \mathsf{K} }}
\newcommand{\coAlg}{{ \mathsf{coAlg} }}

\newcommand{\AlgJ}{{ \Alg_J }}
\newcommand{\coAlgK}{{ \coAlg_\K }}
\newcommand{\AlgO}{{ \Alg_\capO }}

\newcommand{\capO}{{ \mathcal{O} }}

\newcommand{\capR}{{ \mathcal{R} }}

\newcommand{\id}{{ \mathrm{id} }}

\newcommand{\Smash}{{ \,\wedge\, }}

\newcommand{\wequiv}{{ \ \simeq \ }}

\newcommand{\function}[3]{{ {#1}\colon\thinspace{#2}\rightarrow{#3} }}

\DeclareMathOperator*{\holim}{holim}

\DeclareMathOperator{\Hombold}{\mathbf{Hom}}

\title[Homotopy pro-nilpotent structured ring spectra]{Homotopy pro-nilpotent structured ring spectra and topological Quillen localization}

\author{Yu Zhang}

\address{Department of Mathematics, Nankai University, 94 Weijin Road, Nankai District, Tianjin, P.R.China 300071}

\email{zhang.4841@osu.edu}

\subjclass[2020]{13D03; 18M70; 55P43; 55P60}

\thanks{Keywords: (Co)homology of commutative rings and algebras; Algebraic operads and Koszul duality; Spectra with additional structure; Localization and completion in homotopy theory}

\begin{document}

\maketitle

\begin{abstract}

The aim of this paper is to show that homotopy pro-nilpotent structured ring spectra are $\TQ$-local, where structured ring spectra are described as algebras over a spectral operad $\capO$. Here, $\TQ$ is short for topological Quillen homology, which is weakly equivalent to $\capO$-algebra stabilization. An $\capO$-algebra is called homotopy pro-nilpotent if it is equivalent to a limit of nilpotent $\capO$-algebras.  Our result provides new positive evidence to a conjecture by Francis-Gaisgory on Koszul duality for general operads.  As an application, we simultaneously extend the previously known $0$-connected and nilpotent $\TQ$-Whitehead theorems to a homotopy pro-nilpotent $\TQ$-Whitehead theorem. 

\end{abstract}

\section{Introduction}

Spectra play a key role in the development of modern algebraic topology.  Lots of important examples of spectra, such as Eilenberg–Mac Lane spectra, bordism spectra and complex (or real) K-theory spectra, are equipped with natural algebraic structures. However, the algebraic structures on spectra are often more general than their classical analogs, such as commutative rings.
Spectra equipped with generalized algebraic structures are  called structured ring spectra.

We can formalize our definition of structured ring spectra as follows.  Let $\capR$ be any commutative monoid in the category of symmetric spectra of simplicial sets.  In other words, let $\capR$ be any commutative ring spectrum.  Structured ring spectra are spectra with extra algebraic structures that can be described as algebras over an operad $\capO$ in symmetric spectra, or more generally, in $\Mod_{\capR}$.  Here, we let $(\Mod_{\capR},\Smash,\capR)$ denote the symmetric monoidal category of $\capR$-modules.  For a fixed operad $\capO$, denote by $\AlgO$ the category of $\capO$-algebras.  For readers not familiar with operads, \cite{Berger_Moerdijk, Fresse_lie_theory, Fresse_Koszul_duality, Kriz_May, May, Rezk} are some useful references.   In this paper we work with reduced operads $\capO$ (i.e., such that $\capO[0]=*$, where $*$ denotes the trivial $\capR$-module); algebras over $\capO$ are then called non-unital.  This includes the examples of non-unital $E_n$ algebra spectra.

Topological Quillen homology \cite{Basterra, Basterra_Mandell_thh}, or $\TQ$-homology, is the topological analog of André–Quillen homology \cite{Blanc_Johnson_Turner_AQ, Goerss_f2_algebras, Quillen_rings} in the setting of non-unital structured ring spectra. We recall the precise definition of $\TQ$-homology below.

Fix an operad $\capO$ as above.  Let the operad $\tau_1\capO$ be the natural truncation of $\capO$ above level 1.  In particular,  $\tau_1\capO [1] = \capO [1]$ and $\tau_1\capO [k] = *$ for $k \geq 0, k \neq 1$. Then there is a canonical truncation map $\capO \rightarrow \tau_1\capO$ in the category of operads.  We can factor the truncation map as $\capO \rightarrow J \rightarrow \tau_1\capO$, a cofibration followed by a weak equivalence with respect to the projective model structure of operads, see  \cite[Definition 5.47, 7.10]{Harper_Hess} for more details.

The map $\capO \rightarrow J$ induces the
corresponding change of operads adjunction
\begin{equation}
\label{equation: QU adjunction}
\xymatrix{
\AlgO \ar@<.5ex>[r]^Q & \Alg_{J} \ar@<.5ex>[l]^U 
}
\end{equation}
with left adjoint on top, where $Q(X)=J\circ_\capO(X)$ and the forgetful functor $U$ is the restriction along the operad map $\capO \to J$.  Here $\circ$ denotes composition product, $J\circ_\capO(X) = coeq(J \circ \capO \circ X \rightrightarrows  J \circ X)$, see, for example,   \cite[Definition 2.8]{Harper_bar_constructions}.

\begin{convention}
\label{convention: positive flat stable}
Throughout this paper, we work with the positive flat stable model structure (see, for example \cite[7.15]{Harper_Hess}) on $\AlgO$ and $\AlgJ$ unless otherwise specified.  A map of $\capO$-algebras is called a (co)fibration if it is so with respect to the positive flat stable model structure on $\AlgO$.  Similarly, an $\capO$-algebra is called (co)fibrant if it is so with respect to the positive flat stable model structure on $\AlgO$. 
\end{convention}

\begin{rem}
The category $\Alg_{J}$ of $J$-algebras is Quillen equivalent to the category of $\capO [1]$-modules \cite[7.21]{Harper_Hess}.  One can think of $J$ as a fattened-up version of $\tau_1\capO$.  The advantage of working with $J$ instead of $\tau_1\capO$ will become clear after we introduce the $\TQ$-completion construction (Definition \ref{def: tq completion}).  See Remark \ref{rem: J preserve cofibrancy}.
\end{rem}

\begin{defn}\label{def: tq homology}
Let $X$ be an $\capO$-algebra. The \emph{topological Quillen homology} (or $\TQ$-homology, for short) of $X$ is
$$\TQ(X):=\RR U(\LL Q(X))$$
the $\capO$-algebra defined via the indicated composite of total right and left derived functors.  
Note the forgetful functor $U$ preserves all weak equivalences.  Therefore, if $X$ is cofibrant, then $\TQ(X) \simeq UQ(X)$ and the unit of the $(Q,U)$ adjunction in \eqref{equation: QU adjunction} is the $\TQ$-Hurewicz map $X \rightarrow{UQX}$ of the form $X \rightarrow\TQ(X)$.

\end{defn}

$\TQ$-homology has been shown to enjoy several properties analogous to the ordinary homology of spaces; see, for instance, \cite{Basterra, Basterra_Mandell, Harper_Hess}. Furthermore, it turns out that $\TQ$-homology is weakly equivalent to stabilization $\Omega^\infty\Sigma^\infty$ in the category of $\capO$-algebras \cite{Basterra_Mandell, Ching_Harper_derived_Koszul_duality, Kuhn, Pereira_spectral_operad}.  One can think of the adjunction \eqref{equation: QU adjunction} as an analog of suspension spectrum and infinite loop space adjunction $(\Sigma^\infty, \Omega^\infty)$.

Let $K:= QU$ denote the comonad associated to the adjunction $(Q, U)$.  Then the image of $Q$ lands in the category of $K$-coalgebras.  Moreover, there is an associated adjunction of $\infty$-categories \cite[1.3]{Ching_Harper_derived_Koszul_duality}
\begin{equation}
\label{equation: koszul adjunction}
\xymatrix{
\AlgO \ar@<.5ex>[r] & \coAlgK \ar@<.5ex>[l] 
}
\end{equation}
where the left adjoint is $Q$.

Francis-Gaitsgory \cite{Francis_Gaitsgory} studied analogous phenomena in terms of Koszul duality of general operads.  They made a conjecture, which we can rephrase in terms of structured ring spectra as follows.

\subsection*{The Francis-Gaitsgory Conjecture \cite[3.4.5]{Francis_Gaitsgory}}
\textit{Adjunction \eqref{equation: koszul adjunction} induces an equivalence of homotopy categories after restricting $\AlgO$ to the full subcategory of homotopy pro-nilpotent $\capO$-algebras.}

We recall relevant definitions below.

\begin{defn}
\label{defn: homotopy_pro_nilpotent}
Let $X$ be an $\capO$-algebra and $M \geq 2$.  We say that $X$ is \emph{$M$-nilpotent} if all the $M$-ary and higher operations $ \capO[t] \wedge X^{\wedge t} \rightarrow X$ of $X$ are trivial (i.e., if these maps factor through the trivial $\capR$-module ${*}$ for each $t\geq M$).  An $\capO$-algebra is called \emph{nilpotent} if it is $M$-nilpotent for some $M\geq 2$. An $\capO$-algebra is \emph{homotopy pro-nilpotent} if it is weakly equivalent to the homotopy limit of a small diagram of nilpotent $\capO$-algebras.
\end{defn}

Some special cases of the Francis-Gaitsgory conjecture have been proved.

If $\capO$ is truncated, meaning there exists some large enough $n$ such that $\capO [k] = *$ for all $k \geq n$, then the conjecture has been proved by Heuts \cite[6.9]{Heuts_goodwillie_approximation}.  In this special case, all $\capO$-algebras are nilpotent.

For a general operad $\capO$, Ching-Harper \cite[1.2]{Ching_Harper_derived_Koszul_duality} proved that adjunction \eqref{equation: koszul adjunction} induces an equivalence of homotopy categories after restricting to $0$-connected objects on both sides, under the assumption that $\capR$ and $\capO [k]$ for each $k$ are all $(-1)$-connected.  Here we say an $\capO$-algebra $X$ is $0$-connected if the homotopy groups $\pi_k X$ of the underlying spectra are trivial for all $k \leq 0$.

If an $\capO$-algebra $X$ is $0$-connected, then $X$ is homotopy pro-nilpotent.  This is because the homotopy completion tower of $X$ converges strongly to $X$ \cite[1.12]{Harper_Hess}.  Hence, the result of Ching-Harper partially solves the Francis-Gaitsgory conjecture.  The general question for homotopy pro-nilpotent objects remains open.  This is the reason why the main result of \cite{Amabel_koszul} has $0$-connected assumptions.

\begin{rem}
In particular, if one take $\capO$ to be an $E_n$ operad in $\ModR$, the result of Ching-Harper \cite{Ching_Harper_derived_Koszul_duality} is related to the Koszul duality between  $E_n$-algebras and $E_n$-coalgebras, see also \cite{Ayala_Francis,Ching_Salvatore_Koszul,Fresse_Koszul_duality,Lurie_higher_algebra}.
\end{rem}

The unit of adjunction \eqref{equation: koszul adjunction} is shown \cite{Ching_Harper_derived_Koszul_duality} to be weakly equivalent to $\TQ$-completion (Definition \ref{def: tq completion}), which is an analog of Bousfield-Kan completion \cite{Bousfield_Kan} of spaces. Hence, for a general operad $\capO$, to prove the ``unit side'' of the Francis-Gaitsgory conjecture amounts to proving for each (cofibrant) homotopy pro-nilpotent $\capO$-algebra $X$, the $\TQ$-completion map $X \rightarrow X^{\wedge}_{\TQ}$ is a weak equivalence of $\capO$-algebras.  The following are results in this direction.

(1) The result of Ching-Harper \cite{Ching_Harper_derived_Koszul_duality} implies for each $0$-connected $\capO$-algebra $X$,  $X \rightarrow X^{\wedge}_{\TQ}$ is a weak equivalence.  Here $\capR$ and $\capO$ are assumed to be $(-1)$-connected.

(2) Ching-Harper \cite{Ching_Harper_nilpotent_Whitehead} proved for nilpotent $\capO$-algebra $X$, $X$ is a retract of $X^{\wedge}_{\TQ}$ in the homotopy category of $\capO$-algebras.

(3) Schonsheck \cite{Schonsheck_Fibration} proved that if $X$ is the homotopy fiber of a fibration $E \rightarrow B$ of $\capO$-algebras where both $E, B$ are $0$-connected, then $X \rightarrow X^{\wedge}_{\TQ}$ is a weak equivalence.  Here $\capR$ and $\capO$ are assumed to be $(-1)$-connected.

However, none of the known results could work for arbitrary
homotopy pro-nilpotent $\capO$-algebras.  In this paper, we take a different approach and work with $\TQ$-localization (Definition \ref{defn: tq local and localization}) in place of $\TQ$-completion.  Our main result is the following.

\begin{thm}
\label{thm:main_result}
Let $X$ be a homotopy pro-nilpotent $\capO$-algebra, then an arbitrary fibrant replacement of $X$ in $\AlgO$ is $\TQ$-local.
\end{thm}

\begin{rem}
\label{rem: fibrant replacement}
The appearance of fibrant replacement is due to our definition (Definition \ref{defn: tq local and localization}) that $\TQ$-local $\capO$-algebras are required to be fibrant (with respect to the positive flat stable model structure, see Convention \ref{convention: positive flat stable}).  If such $X$ is already fibrant, then $X$ is $\TQ$-local.
\end{rem}

Our result provides positive evidence to the Francis-Gaitsgory conjecture for arbitrary homotopy pro-nilpotent $\capO$-algebra $X$. Indeed, $X \rightarrow X^{\wedge}_{\TQ}$ is a weak equivalence if and only if (1) $X$ is $\TQ$-local, and (2) $X \rightarrow X^{\wedge}_{\TQ}$ is a $\TQ$-homology equivalence (Proposition \ref{prop: complete iff good and local}).  We have proved the first part for homotopy pro-nilpotent $\capO$-algebras, only the second half remains.

As an application of the main result, we obtain the following homotopy pro-nilpotent $\TQ$-Whitehead theorem that simultaneously extends the previously known $0$-connected and nilpotent $\TQ$-Whitehead theorems \cite{Ching_Harper_nilpotent_Whitehead,Harper_Hess}.

\begin{thm}
\label{thm: homotopy pro nilpotent whitehead}
A map $X\rightarrow Y$ between homotopy pro-nilpotent $\capO$-algebras is a weak equivalence if and only if it is a $\TQ$-homology equivalence.
\end{thm}

There are lots of important examples of $\capO$-algebras that are homotopy pro-nilpotent but are not nilpotent nor $0$-connected.  For example, in the context of Goodwillie calculus, the Taylor tower of the identity functor on $\AlgO$ always converges to homotopy pro-nilpotent $\capO$-algebras \cite[1.14]{Harper_Hess}.  But those $\capO$-algebras are not nilpotent nor $0$-connected in general. See also \cite{clark_goodwillie, Kuhn, Pereira_spectral_operad, Schonsheck_TQ_Taylor} for related discussions.

\subsection*{Organization of the paper}

In Section \ref{sec: defn completion and localization}, we review the basic setup for $\TQ$-completion and $\TQ$-localization.  We also recall the $\TQ|_{\mathrm{Nil}_M}$-completion construction, which will play a key role in our proof of the main result (Theorem \ref{thm:main_result}).

In Section \ref{sec: proof}, we prove Theorems \ref{thm:main_result} and \ref{thm: homotopy pro nilpotent whitehead}.  Along the way, we also discuss the relation between $\TQ$-completion and $\TQ$-localization (Proposition \ref{prop: complete iff good and local}).

\subsection*{Assumptions on the operad $\capO$}

We work in the category $\AlgO$ of algebras over an operad $\capO$ in $\Mod_{\capR}$, the category of $\capR$-modules, where $\capR$ is a commutative monoid in the category of symmetric spectra.  Throughout this paper, we assume that $\capO[0]=*$.  We also make a technical assumption that the natural maps $\capR\rightarrow\capO[1]$ and $*\rightarrow \capO[n]$ are flat stable cofibrations in $\capR$-modules for each $n \geq 0$; see, for instance, \cite[2.1, 6.12]{Ching_Harper_derived_Koszul_duality}.  This is the same cofibrancy condition that also appears in \cite{Ching_Harper_derived_Koszul_duality, Harper_Hess}.  This assumption does not limit the usage of our main result since, up to weak equivalence, any operad $\capO$ can be replaced by one that satisfies such conditions.  We do not need connectivity assumptions on $\capR$ and $\capO$.

\subsection*{Acknowledgments}
The author would like to thank John E. Harper and Niko Schonsheck for inspiring discussions and helpful suggestions.  The author would like to thank Michael Ching, Martin Frankland, Mark W. Johnson and Jérôme Scherer for helpful conversations. The author is grateful to Oscar Randal-Williams for detailed and helpful critical comments on an early draft of this paper. The author would like to thank the anonymous referee for detailed suggestions.   The author was supported in part by the Simons Foundation: Collaboration Grants for Mathematicians \#638247, and by the National Natural Science Foundation of China (No. 11871284; 12001474; 12261091; 12271183).

\section{$\TQ$-completion and $\TQ$-localization}
\label{sec: defn completion and localization}

In this section, we review the definitions of $\TQ$-completion and $\TQ|_{\mathrm{Nil}_M}$-completion.  We also recall the definitions of $\TQ$-localization and the $\TQ$-local homotopy theory on $\AlgO$.

We first recall the $\TQ$-completion construction \cite{Harper_Hess}.

Let  $Z$ be a cofibrant $\capO$-algebra.  Consider the cosimplicial resolution of $Z$ with respect to $\TQ$-homology of the form
\begin{equation}
\xymatrix{
Z \ar[r] & 
(UQ) Z \ar@<0.5ex>[r] \ar@<-0.5ex>[r]  & 
(UQ)^2 Z \ar@<-1.5ex>[l]   \ar@<1.0ex>[r] \ar[r] \ar@<-1.0ex>[r] & 
(UQ)^3 Z \ar@<-2.0ex>[l] \ar@<-3.0ex>[l] \cdots 
}
\end{equation}
in $\AlgO$, denoted $Z \rightarrow \mathbf{C} (Z)$, with coface maps obtained by iterating the $\TQ$-Hurewicz map $\id \rightarrow UQ$ (Definition \ref{def: tq homology}) and codegeneracy maps built from the counit map of the adjunction $(Q,U)$ in the usual way. 
Taking the homotopy limit (over $\Delta$) gives a map \cite{Ching_Harper_derived_Koszul_duality,Harper_Hess} of the form
\begin{equation}
\label{eq: tq_completion}
  Z \rightarrow
  Z^{\wedge}_{\TQ}=\holim\nolimits_\Delta \mathbf{C}(Z)
  \wequiv\holim\nolimits_\Delta\widetilde{\mathbf{C}(Z)}
\end{equation}
in $\AlgO$, where $\widetilde{\mathbf{C}(Z)}$ denotes any functorial fibrant replacement functor $\widetilde{(-)}$ on $\AlgO$ (obtained, for instance,  by running the small object argument with respect to the generating acyclic cofibrations in $\AlgO$) applied to the cosimplicial diagram $\mathbf{C}(Z)$.

\begin{defn}
\label{def: tq completion}
Let $Z$ be a cofibrant $\capO$-algebra.  The \emph{$\TQ$-completion} of $Z$ is the map $Z \rightarrow
  Z^{\wedge}_{\TQ}$ of $\capO$-algebras constructed above. 
\end{defn}

\begin{rem}
\label{rem: J preserve cofibrancy}
The construction of $J$ guarantees that both $U$ and $Q$ preserve cofibrant objects \cite[5.49]{Harper_Hess}. Hence, $UQ(Z) \simeq \TQ(Z), (UQ)^2(Z) \simeq (\TQ)^2 (Z), $ etc. This shows the $\TQ$-completion construction is homotopically well defined; weakly equivalent cofibrant $\capO$-algebras have weakly equivalent $\TQ$-completions.
\end{rem}

Next, we recall the $\TQ|_{\mathrm{Nil}_M}$-completion construction from \cite{Ching_Harper_nilpotent_Whitehead}.  This construction is very similar to the $\TQ$-completion construction.  However, it is only defined for $M$-nilpotent $\capO$-algebras.

For each $n\geq 1$, let $\tau_n\capO$ denote the operad associated to $\capO$ where
\begin{equation}
\label{equation: tau n O}
  (\tau_n \capO)[t] :=
    \begin{cases}
      \capO[t] & \text{for $t \leq n$}\\
      * &  \text{otherwise}
    \end{cases}       
\end{equation}
and consider the associated commutative diagram of operad maps
\cite{Ching_Harper_nilpotent_Whitehead}
\begin{equation}
\xymatrix{
\capO \ar@/_1pc/[dr]  \ar[r] & 
J_n \ar[d]^\sim  \ar[r]  & 
J_1 \ar[d]^\sim\ar@{=}[r]  & J\\
& \tau_n \capO \ar[r] & 
\tau_1 \capO
}
\end{equation}
 where the upper horizontal maps
 are cofibrations of operads, the left-hand and bottom horizontal maps are the natural truncations, and the vertical maps are weak equivalences of operads; for notational simplicity, here we take $J=J_1$. The corresponding change of operad adjunctions have the form
\begin{equation}
\xymatrix{
  \AlgO \ar@<.5ex>[r]^{R_n}  & 
  \Alg_{J_n} \ar@<.5ex>[l]^{V_n} \ar@<.5ex>[r]^{Q_n}  & 
  \Alg_{J} \ar@<.5ex>[l]^{U_n}
}\quad\quad
\xymatrix{
  \AlgO \ar@<.5ex>[r]^{Q}  & 
  \Alg_{J} \ar@<.5ex>[l]^{U}
}
\end{equation}
with left adjoints on top, where $R_n=J_n\circ_\capO(-)$, $Q_n=J\circ_{J_n}(-)$, $Q=J\circ_\capO(-)$, and $V_n,U_n,U$ denote the indicated forgetful functors; in particular, the adjunction on the right is the composite of the adjunctions on the left.

Let $n\geq 1$ and define $M:=n+1$. Let $X$ be a cofibrant $J_n$-algebra and consider the cosimplicial resolution of $X$ with respect to $\TQ|_{\mathrm{Nil}_M}$-homology of the form
\begin{equation}
\xymatrix{
  X \ar[r] &  (U_n Q_n) X \ar@<0.5ex>[r] \ar@<-0.5ex>[r]  &  (U_n Q_n)^2 X  \ar@<-1.5ex>[l]   \ar@<1.0ex>[r] \ar[r] \ar@<-1.0ex>[r] &  (U_n Q_n)^3 X \ar@<-2.0ex>[l] \ar@<-3.0ex>[l] ~ \cdots 
}
\end{equation}
in $\Alg_{J_n}$, denoted $X\rightarrow \mathbf{N}(X)$, with coface maps obtained by iterating the $\TQ|_{\mathrm{Nil}_M}$-Hurewicz map $\id \rightarrow U_nQ_n$ and codegeneracy maps built from the counit map of the adjunction $(Q_n,U_n)$ in the usual way. Applying the forgetful functor $V_n$ gives the diagram $V_n X \rightarrow V_n \mathbf{N} (X)$ of the form
\begin{equation}
\xymatrix{
  V_n X \ar[r] &  
  V_n(U_nQ_n) X\ar@<0.5ex>[r]\ar@<-0.5ex>[r] &  
  V_n(U_nQ_n)^2 X\ar@<-1.5ex>[l]\ar@<1.0ex>[r]
  \ar[r]\ar@<-1.0ex>[r] &  
  V_n (U_nQ_n)^3 X\ar@<-2.0ex>[l]\ar@<-3.0ex>[l]\cdots 
}
\end{equation}
in $\AlgO$.  Taking the homotopy limit (over $\Delta$) gives a map of the form
\begin{equation}
\label{eq:defn_of_TQ_Nil_M_completion}
  V_n X\rightarrow
  X^\wedge_{\TQ|_{\mathrm{Nil}_M}}=
  \holim\nolimits_\Delta V_n\mathbf{N}(X)
  \wequiv\holim\nolimits_\Delta\widetilde{V_n\mathbf{N}(X)}
\end{equation}
in $\AlgO$, where $\widetilde{V_n\mathbf{N}(X)}$ denotes any functorial fibrant replacement functor $\widetilde{(-)}$ on $\AlgO$ applied to the cosimplicial diagram $V_n\mathbf{N}(X)$.

\begin{defn}
\label{def: nilpotent completion}
Let $Z$ be an $M$-nilpotent $\capO$-algebra. Choose a cofibrant $J_n$-algebra $X$ such that $Z$ is weakly equivalent to $V_n X$ as $\capO$-algebras. The \emph{$\TQ|_{\mathrm{Nil}_M}$-completion}  of $Z$ is defined as the map $V_n X\rightarrow X^\wedge_{\TQ|_{\mathrm{Nil}_M}}$.
\end{defn}

\begin{rem}
\label{remark: Jn replace nilpotent}
The existence of $X$ in Definition \ref{def: nilpotent completion} is explained in \cite{Ching_Harper_nilpotent_Whitehead} (see the discussion following \cite[Proposition 2.8]{Ching_Harper_nilpotent_Whitehead}).  Moreover, we can make the choice of $X$ to be functorial, although we do not need the extra property in this paper.
\end{rem}

Next, we recall the definition of $\TQ$-localization, as well as the $\TQ$-local homotopy theory constructed in \cite{Harper_Zhang}.

\begin{defn}
\label{defn: tq acyclic strong cofibration}
Let $\function{f}{A}{B}$ be a map in $\AlgO$. We say that $f$ is a
\begin{itemize}
    \item \emph{$\TQ$-equivalence} if $f$ induces a weak equivalence $\TQ(A)\wequiv\TQ(B)$ on $\TQ$-homology.
    \item \emph{strong cofibration} if $f$ is a cofibration between cofibrant objects.
    \item \emph{$\TQ$-acyclic strong cofibration} if $f$ is a strong cofibration which is also a $\TQ$-equivalence.
    \item \emph{weak $\TQ$-fibration} if $f$ has the right lifting property with respect to every $\TQ$-acyclic strong cofibration.
\end{itemize}
\end{defn}

\begin{defn}
\label{defn: tq local and localization}
An $\capO$-algebra $X$ is called \emph{$\TQ$-local} if (i) $X$ is fibrant in $\AlgO$, and (ii) every $\TQ$-acyclic strong cofibration $A\rightarrow B$ induces a weak equivalence
\begin{align*}
\Hombold(A,X)\xleftarrow{\wequiv}\Hombold(B,X) 
\end{align*}
on mapping spaces in $\sSet$; here we are using the simplicial model structure on $\AlgO$ (see, for instance, \cite{Ching_Harper_derived_Koszul_duality, EKMM,  Goerss_Hopkins_moduli_problems, Goerss_Jardine, Harper_Hess}).
The \emph{$\TQ$-localization} of $X$ is a map $\function{l}{X}{L_\TQ(X)}$ in $\AlgO$ such that (i) $l$ is a $\TQ$-equivalence, and (ii) $L_\TQ(X)$ is $\TQ$-local.
\end{defn}

\begin{prop}\cite[5.14]{Harper_Zhang}
\label{prop: tq local semi model}
The category $\AlgO$ with the three distinguished classes of maps (i) $\TQ$-equivalences, (ii) weak $\TQ$-fibrations, and (iii) cofibrations (Convention \ref{convention: positive flat stable}), has the structure of a (left) semi-model category in the sense of Goerss-Hopkins \cite[1.1.6]{Goerss_Hopkins_moduli_problems}. 
\end{prop}

For us, the main difference of working with the semi-model structure compared to full model structures is that (1) we often need to work with strong cofibrations instead of arbitrary cofibrations, and (2) the factorization axiom for the semi-model structure only provides functorial fibrant replacements for cofibrant objects.

\begin{rem}
The $\TQ$-local homotopy theory only results in a semi-model structure instead of a full model structure because the model structure on $\AlgO$ (recall Convention \ref{convention: positive flat stable}) is almost never left proper, in general (e.g., associative ring spectra are not left proper); see, for instance, \cite[2.10]{Rezk_proper}. 
\end{rem}

The following proposition will be useful for detecting $\TQ$-local $\capO$-algebras.

\begin{prop}\cite[5.16]{Harper_Zhang}
\label{prop: local iff weak fibrant}
An $\capO$-algebra $X$ is $\TQ$-local if and only if the map $X \rightarrow *$ is a weak $\TQ$-fibration.  
\end{prop}

Consequently, the functorial factorization of $\TQ$-local semi-model structure gives functorial $\TQ$-localization for cofibrant $\capO$-algebras \cite[5.17]{Harper_Zhang}.

\section{Homotopy pro-nilpotent $\capO$-algebras are $\TQ$-local}
\label{sec: proof}

In this section, we discuss the relation between $\TQ$-completion and $\TQ$-localization (Proposition \ref{prop: complete iff good and local}).  After that, we will use a similar strategy to study $\TQ|_{\mathrm{Nil}_M}$-completion and show that fibrant nilpotent $\capO$-algebras are $\TQ$-local (Proposition \ref{prop: nilpotent local}).  Then, we can prove the main result (Theorem \ref{thm:main_result}).  As an application, we will also discuss the homotopy pro-nilpotent $\TQ$-Whitehead theorem (Theorem \ref{thm: homotopy pro nilpotent whitehead}).

$\TQ$-localization enjoys most nice properties possessed by general (left) Bousfield localizations.  However, we do want to be careful since the $\TQ$-local structure (Proposition \ref{prop: tq local semi model}) is only a semi-model structure instead of a full model structure. We list some useful properties below.  Some good references for general localization techniques include \cite{Bousfield_localization_spaces, Bousfield_Kan,Dror_Farjoun_LNM, Dwyer_localizations,  Hirschhorn, May_Ponto, Sullivan_mit_notes}. 

\begin{prop}
\label{prop: basic props of localization}
(1) A map $X\rightarrow Y$ between $\TQ$-local $\capO$-algebras is a weak equivalence if and only if it is a $\TQ$-homology equivalence.

(2) If $X$ and $Y$ are fibrant $\capO$-algebras that are weakly equivalent, then $X$ is $\TQ$-local if and only if $Y$ is $\TQ$-local.

(3) The homotopy limit of a small diagram of $\TQ$-local $\capO$-algebras is $\TQ$-local.

\end{prop}

\begin{proof}
(1) and (2) are standard facts about localization; see, for instance, Hirschhorn \cite[3.2.13, 3.2.2]{Hirschhorn}.
(3) is also a standard result for left Bousfield localization.  We spell out the details here to show the proof still works when the $\TQ$-local homotopy theory only has a semi-model structure.  

Note the $\TQ$-local semi-model structure has strictly less fibrations compared to the original model structure on $\AlgO$.  It follows from (1) that the homotopy limit in $\AlgO$ of a small diagram of $\TQ$-local $\capO$-algebras is weakly equivalent to its homotopy limit calculated in the $\TQ$-local semi-model structure. Moreover, the diagram is already objectwise fibrant with respect to the $\TQ$-local semi-model structure by Proposition \ref{prop: local iff weak fibrant}.   Hence, the result follows from the fibrancy property of homotopy limits in a homotopy theory (in this case, in the $\TQ$-local homotopy theory); see, for instance, Hirschhorn \cite[18.5.2]{Hirschhorn}, together with Ching-Harper \cite[8.9]{Ching_Harper_derived_Koszul_duality} for a discussion of homotopy limits in the context of $\capO$-algebras.
\end{proof}

For instance,  let $f: X \rightarrow Y$ be a map between $\TQ$-local $\capO$-algebras.  It follows from Proposition \ref{prop: basic props of localization} (3) that the homotopy fiber of $f$ is also $\TQ$-local. This is not expected to be true, in general, if we replace ``$\TQ$-local'' with ``$\TQ$-complete'' (Definition \ref{defn: tq good and tq complete}), and is one of the reasons why $\TQ$-localization is often better behaved than $\TQ$-completion.  See \cite{Schonsheck_Fibration} for related discussions.

The following proposition gives our first examples of $\TQ$-local $\capO$-algebras.

\begin{prop}
\label{prop: u_local}
Let $Y$ be a fibrant object in $\Alg_{J}$.  Then $UY \in  \Alg_{\capO}$ is $\TQ$-local.  Here, $U$ is the right adjoint of adjunction \eqref{equation: QU adjunction}.
\end{prop}

\begin{proof}
By proposition \ref{prop: local iff weak fibrant}, it suffices to show $UY \to *$ has the right lifting property with respect to every $\TQ$-acyclic strong cofibration $i: A \to B$.  Using the $(Q,U)$ adjunction \eqref{equation: QU adjunction}, it is equivalent to show $Y \to *$ has the right lifting property with respect to $Qi: QA \to QB$ in $\Alg_{J}$.  This is certainly true since $Y$ is fibrant and $Qi: QA \to QB$ is an acyclic cofibration in $\Alg_{J}$. 
\end{proof}

The following generalization of Proposition \ref{prop: u_local} will be used in our proof of Propositions \ref{prop: completion is local} and \ref{prop: nil completion local}.

\begin{prop}
\label{prop:fibrant_replacements_of_UY_are_TQ_local} 
Let $Y$ be any object in $\Alg_{J}$, then every fibrant replacement of $UY$ in $\AlgO$ is $\TQ$-local.
\end{prop}

\begin{proof}
Note different fibrant replacements of an object are always related by a zig-zag of weak equivalences.  By Proposition \ref{prop: basic props of localization} (2), it suffices to prove one particular fibrant replacement of $UY$ in $\AlgO$ is $\TQ$-local.  Let $Y'$ be a fibrant replacement of $Y$ in $\AlgJ$.  Then $UY'$ is a fibrant replacement of $UY$.  Now the result follows from Proposition \ref{prop: u_local}.
\end{proof}

\begin{prop}
\label{prop: completion is local}
Let $Z$ be a cofibrant $\capO$-algebra. Then the $\TQ$-completion $Z^\wedge_\TQ$ of $Z$ is $\TQ$-local.
\end{prop}

\begin{proof}
We claim that the $\Delta$-shaped diagram $\widetilde{\mathbf{C}(Z)}$ in \eqref{eq: tq_completion} is objectwise $\TQ$-local; i.e., that $\widetilde{\mathbf{C}(Z)^s}$ is $\TQ$-local for each $s\geq 0$.  Then we can conclude the homotopy limit $Z^\wedge_\TQ$ is $\TQ$-local by Proposition \ref{prop: basic props of localization} (3).

To prove the claimed property, consider $Y:=Q(UQ)^s Z\in\AlgJ$, then $UY=(UQ)^{s+1}Z$. Hence, the fibrant replacement $\widetilde{UY} = \widetilde{\mathbf{C}(Z)^s}$ is $\TQ$-local by Proposition \ref{prop:fibrant_replacements_of_UY_are_TQ_local}.  
\end{proof}

We now discuss the connection between $\TQ$-completion and $\TQ$-localization.

\begin{defn}
\label{defn: tq good and tq complete}
Let $X$ be a cofibrant $\capO$-algebra.  We say $X$ is 
\emph{$\TQ$-good} if the $\TQ$-completion map  $X\rightarrow X^{\wedge}_{\TQ}$ is a $\TQ$-equivalence. We say $X$ is 
\emph{$\TQ$-complete} if the $\TQ$-completion map  $X\rightarrow X^{\wedge}_{\TQ}$ is a weak equivalence.
\end{defn}

\begin{prop}
\label{prop: complete iff good and local}
Let $X$ be a cofibrant $\capO$-algebra.  Then $X$ is $\TQ$-complete if and only if (1) $X$ is $\TQ$-good, and (2) the fibrant replacements of $X$ are $\TQ$-local.
\end{prop}

\begin{proof}
The ``if direction'' follows from Proposition \ref{prop: basic props of localization}(1) and \ref{prop: completion is local}.
The ``only if direction'' follows from Definition \ref{defn: tq good and tq complete} and Proposition \ref{prop: basic props of localization}(2), \ref{prop: completion is local}.
\end{proof}

In the Introduction, we mentioned that the ``unit side'' of the Francis-Gaitsgory conjecture amounts to proving for each cofibrant homotopy pro-nilpotent $\capO$-algebra $X$ that $X$ is $\TQ$-complete.  So far, none of the known results could work for all general homotopy pro-nilpotent objects.  In Theorem \ref{thm:main_result}, we can prove all homotopy pro-nilpotent objects have $\TQ$-local fibrant replacements.  By Proposition \ref{prop: complete iff good and local}, the remaining open question is that whether cofibrant homotopy pro-nilpotent $\capO$-algebras are $\TQ$-good.

We can use a similar strategy to show $\TQ|_{\mathrm{Nil}_M}$-completion also results in $\TQ$-local $\capO$-algebras.

\begin{prop}
\label{prop: nil completion local}
Let $X$ be a cofibrant $J_n$-algebra. Then $X^\wedge_{\TQ|_{\mathrm{Nil}_M}}$ constructed from 
$\TQ|_{\mathrm{Nil}_M}$-completion is $\TQ$-local.
\end{prop}

\begin{proof}
This is similar to the proof of Proposition \ref{prop: completion is local}.  The key observation is that the $\Delta$-shaped diagram $\widetilde{V_n\mathbf{N}(X)}$ in \eqref{eq:defn_of_TQ_Nil_M_completion} is objectwise $\TQ$-local.
\end{proof}

Now we can prove nilpotent $\capO$-algebras are $\TQ$-local up to fibrant replacements.

\begin{prop}
\label{prop: nilpotent local}
Let $Z$ be a nilpotent $\capO$-algebra. Then an arbitrary fibrant replacement of $Z$ in $\AlgO$ is $\TQ$-local.
\end{prop}

\begin{proof}
Let $Z$ be $M$-nilpotent for some $M\geq 2$.   Choose a cofibrant $J_n$-algebra $X$ such that $Z$ is weakly equivalent to $V_n X$ as $\capO$-algebras (with $n = M-1$ as in Remark \ref{remark: Jn replace nilpotent}).  It is proved in Ching-Harper \cite[2.12]{Ching_Harper_nilpotent_Whitehead} that the map $V_n X \to X^\wedge_{\TQ|_{\mathrm{Nil}_M}}$ is a weak equivalence.  Since $Z$ is weakly equivalent to $V_n X$ and $V_n X$ is weakly equivalent to the $\TQ$-local $\capO$-algebra $X^\wedge_{\TQ|_{\mathrm{Nil}_M}}$ (Proposition \ref{prop: nil completion local}), the result follows from Proposition \ref{prop: basic props of localization}(2).
\end{proof}

\begin{proof}[Proof of Theorem \ref{thm:main_result}]
By definition, the homotopy pro-nilpotent $\capO$-algebra $X$ is weakly equivalent to the homotopy limit of a small diagram of nilpotent $\capO$-algebras.  By taking objectwise fibrant replacements for the small diagram, $X$ is weakly equivalent to the homotopy limit of a small diagram of $\TQ$-local $\capO$-algebras (Proposition \ref{prop: nilpotent local}).  Now the result follows from Proposition \ref{prop: basic props of localization}(2)(3).
\end{proof}

As a corollary, we obtain the homotopy pro-nilpotent $\TQ$-Whitehead theorem.

\begin{proof}[Proof of Theorem \ref{thm: homotopy pro nilpotent whitehead}]

We take a functorial fibrant replacement $R$ as follows:
\begin{equation}
\xymatrix{
 X \ar[d]_{\sim} \ar[r]^{f} & Y \ar[d]^{\sim}  \\
 RX \ar[r]^{Rf} & RY   \\
}
\end{equation} 
Then $f$ is a weak equivalence (resp. $\TQ$-homology equivalence) if and only if $Rf$ is a weak equivalence (resp. $\TQ$-homology equivalence).  By Theorem \ref{thm:main_result}, $RX, RY$ are $\TQ$-local.  Then Proposition \ref{prop: basic props of localization} (1) completes the proof.
\end{proof}

\begin{rem}
Here in the proof of Theorem \ref{thm: homotopy pro nilpotent whitehead}, the functorial fibrant replacement functor $R$ is taken with respect to the positive flat stable model structure on $\AlgO$ (see Convention \ref{convention: positive flat stable} and Remark \ref{rem: fibrant replacement}).  This is a (full) model structure, hence we do not need to assume $X,Y$ are cofibrant.  On the contrary, additional cofibrancy conditions might be necessary if one works with the $\TQ$-local semi-model structure (see the discussion following Proposition \ref{prop: tq local semi model}).
\end{rem}

Previously, $\TQ$-Whitehead theorems have been established for $0$-connected and nilpotent $\capO$-algebras separately \cite{Ching_Harper_nilpotent_Whitehead,Harper_Hess}.  However, if one considers a map $X \rightarrow Y$ from a $0$-connected $\capO$-algebra to a nilpotent $\capO$-algebra, then none of those $\TQ$-Whitehead theorems could apply.  Now, $\TQ$-Whitehead theorem becomes applicable to $X \rightarrow Y$ since $0$-connected $\capO$-algebras and nilpotent $\capO$-algebras are all homotopy pro-nilpotent \cite[1.12]{Harper_Hess}.

We also want to point out that Goodwillie calculus \cite{Goodwillie_calculus_3, Kuhn_survey} provides a class of naturally occurring examples that are homotopy pro-nilpotent but are, in general, not $0$-connected nor nilpotent.  

As explained in \cite[1.14]{Harper_Hess}, up to weak equivalence, the Taylor tower for a cofibrant $\capO$-algebra $X$ has the following form:
\begin{equation}
\xymatrix{
 & \vdots \ar[d]  \\
  & \tau_3\capO \circ_\capO X \ar[d]  \\
 & \tau_2\capO \circ_\capO X \ar[d]  \\
X \ar[r] \ar[ur] \ar[uur] & \tau_1\capO \circ_\capO X 
}
\end{equation}
where $\tau_k\capO$ is the operad defined in \eqref{equation: tau n O}. By definition, $\tau_k\capO \circ_\capO X$ regarded as an $\capO$-algebra is $(k+1)$-nilpotent.
Therefore, the Taylor tower of the identity functor on $\AlgO$ always converges to homotopy pro-nilpotent $\capO$-algebras. 
Also see \cite{clark_goodwillie, Kuhn, Pereira_spectral_operad, Schonsheck_TQ_Taylor} for related discussions.

\bibliographystyle{plain}
\bibliography{HomotopyProNilpotent}

\end{document}